\documentclass{article}
\usepackage[utf8]{inputenc}
\usepackage[dvips]{graphicx}
 \usepackage{amssymb,amsmath}
 \usepackage{amsthm}
\usepackage{epic}
\usepackage{pstricks}
\usepackage{pst-node}
\usepackage{color}

\usepackage{tikz}
\usetikzlibrary{arrows,decorations.markings}
\usepackage[utf8]{inputenc}
\usepackage{pstricks-add}

\newtheorem{theorem}{Theorem}
\newtheorem{proposition}[theorem]{Proposition}
\newtheorem{lemma}[theorem]{Lemma}

\newtheorem{definition}[theorem]{Definition}

\date{}

\begin{document}

\title{Geometric constructions of small regular graphs with girth 7}

\author
{Gy\"orgy Kiss\footnote{This research was supported in part  by the Hungarian National Research, Development and Innovation Office  OTKA grant no. SNN 132625.}}

\maketitle

\begin{abstract}
We present simple, geometric constructions for small regular graphs of girth 7 from the incidence graphs of some generalized quadrangles. We obtain infinite families of $(q-1)$-regular, $q$-regular and $(q+1)$-regular graphs of girth 7, for $q$ a prime power. Some of them have the smallest order known so far. 
\end{abstract}


\section{Introduction, definitions and basic properties}

The cage problem is a classical problem in extremal graph theory. In a regular graph of degree $k$ and girth $g$ we ask for the minimum number of vertices. 
Here we briefly summarize the most important results related to the problem.
For more details on cages we refer to the dynamic survey by Exoo and Jajcay \cite{ExooJaj08}.

A $(k,g)$-graph is a $k$-regular graph with girth $g.$ A \emph{$(k,g)$-cage} is a 
$(k,g)$-graph of minimum order. It was proved by Erd\H os and Sachs 
\cite{ErdSachs63} that $(k,g)$-cages exist for every $k\geq 2$ and $g\geq 3.$ The order of a $(k,g)$-cage is denoted by $n(k,g).$ For $k=2,$ the $(2,g)$-graphs are the $g$-cycles, so $n(2,g)=g.$ From now on, we assume that $k\geq 3.$ 
A general lower bound on $n(k,g),$ known as the Moore bound and denoted by 
$M(k,g)$, depends on the parity of $g.$ It is obtained by counting the vertices whose distance from a vertex (if $g$ is odd),
or an edge (if $g$ is even) is at most $\lfloor (g-1)/2\rfloor $.  
$$n(k,g)\geq M(k,g)=\left\{ 
\begin{array}{ll}
1+\sum_{i=0}^{(g-3)/2}k(k-1)^i=\frac{k(k-1)^{(g-1)/2}-2}{k-2}, 
& \text{if $k$ is odd,} \\
2\sum_{i=0}^{(g-2)/2}(k-1)^i=\frac{2(k-1)^{g/2}-2}{k-2},  & \text{if $k$ is even,}   
\end{array}
\right. $$
Graphs attaining this bound 
are called Moore graphs. For $g=3$ and $4,$ Moore graphs are the complete, and complete bipartite graphs, respectively. Moore graphs are rare  for $g>4$. When $g$ is even, then
there exists a $k$-regular Moore graph with girth $2r>4$ if and only if there exists a
finite generalized $r$-gon of order $(k-1,k-1),$ and the graphs are the incidence graphs of the generalized polygons. When $g>4$ odd, then there exist Moore graphs only for $g=5$ and $k=2,3,7$ and possibly $57.$

Several different constructions of small $(k,g)$-graphs were presented in the last few decades. In most cases, $g=6,8$ or $12$ and the graph is related to the incidence graph of a generalized $g/2$-gon. A comprehensive overview of geometric constructions can be found in the paper by Dam\'asdi, H\'eger and Sz\H onyi \cite{DamHegSzo}. 

For $g=5,$ some classes of graphs were constructed by careful manipulations
of the incidence graphs of finite projective planes. Some vertices are removed from 
these $(q+1,6)$-cages, after that matchings or cycles are added to the neighbours of
the removed vertices to get back regularity. The first construction due to Brown \cite{Br-5} who proved that if $q\geq 5$ is a prime power, then $n(q+2,5)\leq 2q^2.$ More sophisticated methods were presented by Abreu et al. \cite{abr+3}, and 
Abajo and her co-authors in a series of papers
\cite{ab-1, ab-2, ab-3}. The exact orders of their $(k,5)$-graphs are a bit
smaller than $2k^2$, but the magnitude is still  $2k^2$.
Recently, Araujo-Pardo and Leemans \cite{AL22} and Araujo-Pardo, Kiss and 
Porups\'anszki \cite{AKP} presented a very simple geometric 
construction for $(q+2,5)$-graphs whose order is $2q^2-2$. 

For $g=7,$ only three infinite families of $(k,7)$-graphs
whose order is approximately twice the Moore bound (which is
$k^3-k^2+k+1$) were known. A more general construction of Lazebnik, Ustimenko and Woldar \cite{LUW} results in $(q,7)$-graphs of order $2q^3,$ and $(q-1,7)$-graphs of order $2q^3-2q^2$ for each odd prime power $q>3.$ The smallest $(q+1,7)$-graphs known so far, were presented by Abreu et al. \cite{abr+4}. They proved that
$$n(q+1,7)\leq \left\{ 
\begin{array}{ll}
2q^3+q^2+2q, & \text{if $q\geq 4$ is an even prime power,} \\
2q^3+2q^2-q+1, & \text{if $q\geq 5$  is an odd prime power.}  
\end{array}
\right. $$

In this paper, constructions of $(k,7)$-graph for $k=q-1,q,q+1$ are given for each prime power $q>3.$  The main idea is the same as described in the previous paragraph, 
but our method is purely geometric. We choose a beautiful geometric substructure of a suitable generalized quadrangle. After that delete the set of the corresponding vertices from the incidence graph. Finally, again using the geometric properties, define some new edges. For $k=q-1$ and $k=q$ and large $q,$ these graphs are the smallest $(k,7)$-graphs known so far. 
Comparing the orders of the $(q+1,7)$-graphs presented in Theorems \ref{main} and
\ref{paros-jobb} and those constructed by Abreu et al., they are the same if $q$ is even, and the orders of the ones we made are slightly larger if $q$ is odd. However, our construction can be described much more briefly and is much simpler.

For a detailed introduction to generalized polygons and the concepts from finite geometries we use can be found for example in \cite{KSz} and \cite{VM98}. 
In particular, for generalized quadrangles we also refer the book \cite{PT}.
Here we give only the most necessary definitions.
 
A finite \emph{generalized quadrangle} of order $(s,t)$, 
shortly a $GQ(s,t),$ is a point-line incidence structure which satisfies the following axioms:
\begin{itemize}
\item for any two distinct points, there is at most one line incident with both;
\item every line is incident with exactly $s+1$ points; 
\item every point is incident with exactly $t+1$ lines;
\item for every point $P$ and every line $\ell$ not incident with $P$, there is exactly one line through $P$ which intersects $\ell$.
\end{itemize}
The \emph{incidence graph}, also called as the \emph{Levi graph}, of a point-line incidence structure is a bipartite graph whose bipartition sets correspond to the
set of points and set of lines, respectively, and there is an edge between two vertices if
and only if the corresponding point is incident with the corresponding line.  Throughout this paper, if $G$ is an incidence graph, then its vertices will be called points and lines, 
and the type of a vertex is either point or line, according to whether they correspond to a point or a line of the geometric structure.

\section{The constructions}

We use four families of generalized quadrangles. These are $\mathcal{W}(q),$
$\mathcal{T}_2(\mathcal{O}),$ $\mathcal{T}_2^*(\mathcal{H}),$ and 
$\mathcal{Q}(4,q)$. Each of them is a so-called projective GQ, which means that there exists a finite projective space 
$\mathcal{B}=\mathrm{PG}(d,q)$ for which the point-set of the GQ is a subset of the set of points of $\mathcal{B}$, the line-set of the GQ is a subset of the set of lines of 
$\mathcal{B}$, and the incidence in the GQ is the one inherited from $\mathcal{B}.$ 
We always consider $\mathrm{PG}(d,q)$ as 
$\mathrm{AG}(d,q)\cup \Sigma_{\infty }.$ The points of $\mathrm{PG}(d,q)$
are coordinatized by homogeneous coordinates $(x_0:x_1:\dots :x_d)$ in the usual way.
The equation of $\Sigma_{\infty }$ is $X_0=0.$ Points and lines in $\Sigma_{\infty }$ 
are called points at infinity and lines at infinity, respectively. If $d=3,$ then
in $\mathrm{AG}(3,q)$ we also use the Cartesian coordinates, $X=X_1/X_0,$ 
$Y=X_2/X_0$ and $Z=X_3/X_0.$ We say that a line or a plane is horizontal, if it is
parallel with the plane $Z=0.$

First, we present some statements that are fulfilled in all GQs.
\begin{definition}
Let $G=(V,E)$ be a $t$-regular induced subgraph of the Levi graph of a $GQ(s,s).$ 
A subset of vertices $W\subset V$ is called \emph{deleteable} if it has the following 
two properties:
\begin{itemize}
\item
if $W$ contains two intersecting lines whose point at intersection is $P,$ then 
$W$ contains $P$ and all the $t$ lines through $P;$
\item
if $W$ contains two collinear points whose joining line is $\ell ,$ then 
$W$ contains $\ell $ and all the $t$ points on the line $\ell .$
\end{itemize}
\end{definition}

The proof of the following proposition is straightforward.
\begin{proposition}
\label{parositashoz}
Let $G=(V,E)$ be a $t$-regular induced subgraph of the Levi graph of a $GQ(s,s),$
and $W\subset V$ be a deleteable subset. 
Then the subgraph $G'$ of $G$ induced by the set of vertices $V\setminus W$
contains two types of vertices. If a vertex has a neighbour in $W,$ then its valency 
in $G'$ is $t-1$, while if a vertex has no neighbour in $W,$ then its valency  in $G'$ is 
$t.$
If $t$ is odd, then the set of the valency $t-1$ vertices of $G'$ admits perfect matchings such that any new edge joins either two points on a deleted line, or two lines through a deleted point.  
\end{proposition}

\begin{lemma}
\label{legalabb6}
Let $G=(V,E)$ be a $t$-regular induced subgraph of the Levi graph of a $GQ(s,s)$ 
where $s+1\geq t$ and $t$ is odd. Let $W$ be a deleteable subset of vertices of $G.$ 
Let $G'$ be the subgraph of $G$ induced by the set of vertices $V\setminus W.$
Add a perfect matching of the valency $t-1$ vertices of $G'$ such that any new edge joins either two points on a deleted line, or two lines through a deleted point. Then the new graph $\Gamma $ is $t$-regular and its girth is at least 6.

Moreover, if $\mathcal{C}$ is a $6$-cycle in $\Gamma ,$ then $\mathcal{C}$
contains exactly two new edges and $\mathcal{C}$ can be given either as
$u_1v_2v_3u_4v_5v_6$ or $u_1v_2v_3u_4u_5v_6$ where the vertices denoted by $u_i$ are of the same type and the vertices denoted by $v_j$ are of the other type.
\end{lemma}

\begin{proof}
By Proposition \ref{parositashoz}, the required perfect matching exists, there is exactly one new edge through each neighbour of a deleted vertex, hence $\Gamma $ is 
$t$-regular.
 
Suppose that $\Gamma $ contains a cycle $\mathcal{C}$ of length $\ell \leq 6$. 
The girth of $G$ is $8,$ so $\mathcal{C}$ must contain some new edges. There is at most one new edge through any vertex of $\Gamma ,$ so $\mathcal{C}$ cannot contain
3 consecutive vertices of the same type.

Hence if $\ell \leq 5,$ then $\mathcal{C}$
contains at most 2 new edges. If $xy$ is a new edge of $\Gamma ,$ then, by definition,
there exists a vertex $z\in W$ such that both $xz$ and $zy$ are edges
in $G.$ Changing the the new edge(s) of $\Gamma $ to path(s) of length 2 in $G,$ the $\ell $-cycle $\mathcal{C}$ results in a cycle $\mathcal{C}'$ of length 
$\ell '\leq \ell +2\leq 7$ in $G,$ a contradiction.

Now, let $\mathcal{C}=u_1v_2v_3u_4w_5w_6$ be a $6$-cycle in $\Gamma .$  
We may assume without loss of generality that the vertices denoted by $v_i$ are of the same type, hence the vertices denoted by $u_j$ are of the other type. If $w_5$ and $u_4$ are
the same type, then $w_6$ must be other type. So $\mathcal{C}$ contains the two new edges $v_2v_3$ and $u_4w_5,$ one of the new edges joins two points, the other
joins two lines. 
If $w_5$ and $u_4$ are different types and $w_6$ is the same type as $w_5,$ then
$\mathcal{C}$ contains the two new edges edges $v_2v_3$ and $w_5w_6.$ The four 
vertices on the two new edges are of the same type. Finally, if $w_5$ and $u_4$ are of different types and $w_6$ is the same type as $u_4,$
then $\mathcal{C}$ contains the two new edges edges $v_2v_3$ and $w_6u_1.$ 
In this case one of the new edges joins two points, the other joins two lines. 
\end{proof}

\begin{lemma}
\label{legalabb6-2}
Let $G$ be a $k$-regular induced subgraph of the Levi graph of a generalized 
quadrangle $\mathcal{Q}$ of order $(s,s).$ 
Suppose that $k\geq 7$ and $\mathcal{Q}$ contains a $k\times k$ grid $\mathcal{G}.$ 
Let $\{ e_1,e_2,\dots ,e_k\} $ and  $\{ f_1,f_2,\dots ,f_k\} $ denote the
two sets of pairwise non-intersecting lines of $\mathcal{G},$ and let $P_{i,j}$
be the point of intersection of the lines $e_i$ and $f_j.$
Delete the $2k$ lines of $\mathcal{G}$ from $G.$  After that add
$k\times k$ new edges joining the pairs of points $(P_{i,j},P_{i,j+1})$ and  
$(P_{i,k},P_{i,1})$ for all $i=1,2,\dots ,k$ and $j=1,2\dots , k-1.$ Denote the new graph by $\Gamma $. Then $\Gamma $ is $k$-regular and its girth is at least 7.
\end{lemma}

\begin{proof}
The proof is similar to the proof of the previous lemma. 

For any fixed $i,$ the new edges among the vertices $P_{i,m}$ form a cycle of length
$k.$ So there are exactly two new edges through each point of $\mathcal{G},$ hence 
$\Gamma $ is $k$-regular.

Suppose that $\Gamma $ contains a cycle $\mathcal{C}$ of length $\ell \leq 6$. 
The girth of $G$ is $8,$ so $\mathcal{C}$ must contain some new edges. Since $k>6,$
therefore $\mathcal{C}$ cannot consists of only new edges. If $u_1u_2\dots u_i$ is the only arc of
consecutive new edges in $\mathcal{C}$, then $u_1$ and $u_i$ are collinear points in 
$\mathcal{Q}$, both are incident with a line $v.$ Thus replacing $u_1u_2\dots u_i$
by $u_1vu_i$ in $\mathcal{C}$ results in a
cycle $\mathcal{C}'$ of length $\ell '\leq \ell +1\leq 7$ in $G,$ a contradiction.
Hence $\mathcal{C}$ must contain at least two non-consecutive new edges.
This implies that $\mathcal{C}$ is a $6$-cycle, and it consists of the vertices
$u_1u_2v_3u_4u_5v_6$ where each $u_i$ denotes a point and each $v_j$ 
denotes a line. So, by the definition of the new edges, $u_2$ and $u_4$ are points on two non-intersecting lines of $\mathcal{G}.$ If $u_2=P_{i,m}$ and $u_4=P_{j,n},$
with $i\neq j,$ then for $m\neq n$ they are not collinear in $G$, while for $m=n$
they are on the deleted line $f_m.$ In both cases, 
they do not have a common neighbour in $\Gamma .$

This contradiction proves the statement.
\end{proof}

Let $k>2$ be an even integer and consider 
the $k\times (k-1)$ part of the plane lattice of integers  
$$\mathcal{G}=\{ (i,j)\colon 1\leq i\leq k-1, \, 1\leq j\leq k\} .$$ 

\begin{definition}
A perfect matching of the set of vertices of $\mathcal{G}$ is \emph{rectangle-free,}
if the four vertices of no two edges form a rectangle on the plane.
\end{definition}

\begin{lemma}
\label{rfree}
 $\mathcal{G}$  admits a rectangle-free matching. 
\end{lemma}
\begin{proof}
Label the vertices of the complete graph $K_k$ by the natural numbers 
$1,2,\dots ,k$ and let
$\mathcal{F}=\{ F_1,F_2,\dots ,F_{k-1}\} $ be a one-factorization of $K_k$.
Define a matching on $\mathcal{G}$ in the following way.
The points $(i_1,j_1)$ and $(i_2,j_2)$ form a pair if and only if
$i_1=i_2$ and the edge $j_1j_2$ of $K_k$ belongs to the one-factor $F_{i_1}.$

By definition, this matching contains $k(k-1)/2$ edges, all of them are vertical. The edges on a fixed vertical line of the lattice (say the line with equation $X=i$), correspond to the $k/2$ edges contained in a fixed one-factor of $\mathcal{F}$ (in particular, $F_i)$. Each one-factor contains all vertices of  $K_k,$ so the matching is perfect. If there were a rectangle in the matching, say 
$(i_1,j_1)(i_1,j_2)$ and $(i_2,j_1)(i_2,j_2),$ then the edge $j_1j_2$ of $K_k$ 
would belong to both one-factors $F_{i_1}$ and $F_{i_2},$ a contradiction.
\end{proof}

\begin{lemma}
\label{construct}
Let $G$ be a $k$-regular induced subgraph of the Levi graph of a generalized quadrangle 
$\mathcal{Q}.$ Suppose that $k$ is even and $\mathcal{Q}$ contains a 
$k\times (k-1)$ grid. Let $\{ e_1,e_2,\dots ,e_k\} $ and  $\{ f_1,f_2,\dots ,f_{k-1}\} $ denote the two sets of pairwise non-intersecting lines of $\mathcal{G},$ and let $P_{i,j}$
be the point of intersection of the lines $e_i$ and $f_j.$
Delete the $k-1$ pairwise non-intersecting 
lines of the grid and let $G'$ denote the incidence graph of the remaining structure.
Identify the points $P_{i,j}$ with the lattice points $(i,j)$ and consider the rectangle-free perfect matching constructed in Lemma \ref{rfree}. Add the corresponding new edges 
to $G'$. Let $\Gamma $ denote the graph $G'$ extended by these new edges. Then 
$\Gamma $ is a $(k,g)$-graph with $g\geq 7.$
\end{lemma}

\begin{proof}
There is one new edge through each point of $G'$, so $\Gamma $ is $k$-regular.
Any new edge joins two points on a deleted line, so we can apply Lemma \ref{legalabb6}. This gives that the girth of $\Gamma $ is at least 6, and any $6$-cycle contains two
new edges.

Suppose that $\mathcal{C}$ is a $6$-cycle. Both new edges of $\mathcal{C}$ join
two points of $G'$, so again by Lemma \ref{legalabb6},
$\mathcal{C}=(a,B,C,d,E,F)$ where uppercase letters denote points of $G'$.
This means that $B,$ $C$, $E$ and $F$ form a rectangle in $G'$, contradicting to the fact that we defined a rectangle-free matching.
\end{proof}

In the following theorems, we construct small regular graphs using the unique, specific properties of each GQ.

\begin{theorem}
\label{main}
Let $q\geq 7$ be a prime power. Then
$$n(k,7)\leq \left\{ 
\begin{array}{ll}
2k^3-4k^2+2k, & \text{if $k=q+1$,} \\
2k^3-2k, & \text{if $k=q$.}  
\end{array}
\right. $$
\end{theorem}

\begin{proof}
The starting point of the construction is a $k$-regular induced subgraph $G$
of the Levi graph of a GQ in both cases.
We apply Lemma \ref{legalabb6-2}, and show that the girth of the extended graph $\Gamma $ is exactly $7.$

$(i)$: If $k=q+1$, then take a parabolic quadric $\mathcal{P}$ in 
$\mathrm{PG}(4,q).$ The points and lines entirely contained in $\mathcal{P}$ form the
classical generalized quadrangle $\mathcal{Q}(4,q)$, whose order is $(q,q)$. 
Let $\Sigma $ be a hyperplane such that
$\mathcal{P}\cap \Sigma $ is a hyperbolic quadric $\mathcal{H}$ (e.g: if $\mathcal{P}$ has equation $X_0^2+X_1X_2+X_3X_4=0,$
then the equation of $\Sigma $ could be $X_0=0$).
Then $\mathcal{P}$ contains the $2(q+1)$ lines and $(q+1)^2$ points of 
$\mathcal{H}$, and these points and lines form a $(q+1)\times (q+1)$ grid.
Let $G$ be the Levi graph of $\mathcal{Q}(4,q)$, take $\mathcal{G}=\mathcal{H}$ 
and apply Lemma \ref{legalabb6-2}. The constructed graph $\Gamma $ is a $(q+1)$-regular graph on 
$$2q^3+2q^2=2(q+1)^3-4(q+1)^2+2(q+1)$$ 
vertices and its girth is at least $7$.
The order of $\Gamma $ is less than $M(q+1,8)=2(q^3+q^2+q+1),$ so its girth
is $7.$

$(ii)$: If $k=q$, then take 
$\mathrm{PG}(3,q)=\mathrm{AG}(3,q)\cup \Sigma_{\infty }$ 
and let $\mathcal{A}$ be a $q$-arc in $\Sigma_{\infty }.$ 
We define a point-line incidence geometry $\mathcal{S}=(\mathcal{P},\mathcal{L},
\mathrm{I})$ in the following way. The points of 
$\mathcal{S}$ are the points of $\mathrm{AG}(3,q),$ the
lines of $\mathcal{S}$ are those affine lines of $\mathrm{PG}(3,q)$
whose point at infinity belongs to $\mathcal{A},$ and the incidence is the same as in  
$\mathrm{AG}(3,q).$ Let $G$ denote the Levi graph of $\mathcal{S}.$

When $q$ is odd, then there exists a point $P$ in $\Sigma_{\infty }$ such that 
$\mathcal{O}=\mathcal{A}\cup \{ P\} $ is an oval. This oval defines the 
GQ $\mathcal{T}_2(\mathcal{O}),$ and by definion, $G$ is a $q$-regular 
induced subgraph of the Levi graph of $\mathcal{T}_2(\mathcal{O}).$
Similarly, when $q$ is even, then there exist two points $P$ and $R$ in 
$\Sigma_{\infty }$ such that $\mathcal{H}=\mathcal{A}\cup \{ P,R\} $ is a hyperoval. 
This hyperoval defines the 
GQ $\mathcal{T}_2^*(\mathcal{H}),$ and by definion, $G$ is a $q$-regular 
induced subgraph of the Levi graph of $\mathcal{T}_2^*(\mathcal{H}).$
In both cases, let $AB$ be a secant line of $\mathcal{A}$ and $\Pi $ be an affine plane whose line at infinity is $AB.$ 
The $q$ lines of $\mathcal{S}$ through $A$ form a class of parallel lines in $\Pi ,$
and so do the $q$ lines of $\mathcal{S}$ through $B,$ too. Hence the $2q$ lines of 
$\mathcal{S}$ in $\Pi $ form a $q\times q$ grid.  

We can apply Lemma \ref{legalabb6-2}. The constructed graph $\Gamma $ is a 
$q$-regular graph on $2q^3-2q$ vertices and its girth is at least $7$. We show 
that $\Gamma $ contains $7$-cycles. Let $C,$ $D,$ $E$ be three distinct points
of $\mathcal{A}\setminus \{ A,B\} .$ Since $q\geq 7 ,$ these points exist.
Take two affine points $U_1$ and $U_2$ on a line through $A.$ Then 
$C,\, E$ and $D$ are not collinear, hence the plane
$\langle U_1,C,E \rangle $ intersects the line $U_2D$ in a single point $U_4.$
The lines $U_1C$ and $U_4E$ are coplanar, so they have a unique affine point of intersection, $U_6.$ Let $u_i$ denote the vertex in $\Gamma $ that corresponds to the point
$U_i,$ and let $v_3, \, v_5$ and $v_7$ be the vertices corresponding to the lines
$U_2U_4,\, U_4U_6$ and $U_6U_1,$ respectively. Then, by definition, $u_1u_2v_3u_4v_5u_6v_7$ is a $7$-cycle in $\Gamma .$ 
\end{proof}

We can combine Part $(i)$ of the previous theorem and Lemma \ref{construct}. This results in larger graphs than the previous theorem, but it works also for $q=3$ and 
$5.$

\begin{theorem}
Let $q\geq 3$ be an odd prime power. Then
$$n(q+1,7)\leq 2q^3+2q^2+q+2.$$
\end{theorem}

\begin{proof}
Let $\mathcal{P},$  $Q(4,q),$ $\mathcal{H}$ and $G$ be the same as in the 
proof of Theorem \ref{main}. Let $\ell $ be a line that is contained in 
$\mathcal{H}$. Then the lines of $\mathcal{H}$ except $\ell $ form a 
$(q+1)\times q$ grid $\mathcal{G}.$
Now, apply Lemma \ref{construct}. The constructed graph $\Gamma $ is a $(q+1)$-regular graph on $2q^3+2q^2+q+2$ vertices and its girth is at least $7$.
The order of $\Gamma $ is less than $M(q+1,8)=2(q^3+q^2+q+1),$ so its girth
is $7.$  
\end{proof}

The construction given in Part $(iii)$ of Theorem \ref{main} can 
be extended from $q-1$ to all even numbers $k<q.$
This gives a slight improvement on the general upper bound of Lazebnik,
Ustimenko and Woldar \cite{LUW}.

\begin{theorem}
Let $k\geq 4$ be an even integer and let $q$ denote the smallest prime power
for which $k\leq q.$ Then
$$n(k,7)\leq 2kq^2-q.$$

In particular, if $q$ is odd, then
$$n(q-1,7)\leq 2q^3-2q^2-q=2(q-1)^3+4(q-1)^2+(q-1)-1.$$
\end{theorem}

\begin{proof}
The construction is similar to the one presented in the proof of Theorem \ref{main}, Part $(ii).$ Consider $\mathrm{PG}(3,q)$ as $\mathrm{AG}(3,q)\cup \Sigma _{\infty }.$ Take an oval $\mathcal{O}$ in $\Sigma _{\infty }$. If $q$ is even, then extend it to a hyperoval $\mathcal{H}=\mathcal{O}\cup \{ N\} $ where $N$ denotes the nucleus of 
$\mathcal{O}.$ Choose a  line $e$ in $\Sigma _{\infty }$ that is external to 
$\mathcal{O},$ and does not contain $N$ if $q$ is even. Pick up four points 
$C,\, D,\, M_1, \, M_2$ of $\mathcal{O}.$  Take a line $\ell $ through $C$ and two
affine points $U_1,\, U_2$ on $\ell .$ For $i=1,2$ let $f_i$ denote
the line $U_iM_i.$ Then $f_1$ and $f_2$ are skew lines in  
$\mathrm{PG}(3,q)$ and none of them contains $C.$  
Let $g$ be the line of intersection of the two planes $\langle f_i,C\rangle .$ Then 
$g$ and $f_i$ are coplanar, so there exists the affine point of intersection 
$R_i=f_i\cap g.$ 
Choose $q-k$ parallel affine planes $\Sigma _1,\Sigma _2,\dots ,\Sigma _{q-k}$ such that for all $j=1,2,\dots ,q-k$ the line at infinity of $\Sigma _j$ is $e,$ and
$\Sigma _j\cap \{ U_1,U_2,R_1,R_2\} =\emptyset .$
Then each affine line whose point at infinity belongs to $\mathcal{O}$ intersects
$\Sigma _j$ in a unique affine point. For $q$ even, it is also true for the affine
lines whose point at infinity is $N.$ 
Finally, choose $q+1-k$ points $F_1,F_2,\dots ,F_{q+1-k}$ points of 
$\mathcal{O}\setminus \{ C,D,M_1,M_2\} .$ Since $4\leq k<q,$ the planes 
$\Sigma _j$ and the points $F_j$ exist.

Change the definition of $\mathcal{S}$ such that the set of its points consists of the 
$kq^2$ points of 
$$\mathrm{AG}(3,q)\setminus \left( \bigcup_{j=1}^{q-k}\Sigma _j \right) ,$$ 
and the lines of
$\mathcal{S}$ are those affine lines whose point at infinity belongs to the set
$$\mathcal{O}\setminus \{ F_1,F_2,\dots ,F_{q+1-k}\} .$$ 
Then the Levi graph 
$G$ of $\mathcal{S}$ is a $k$-regular induced subgraph of the Levi graph of 
$\mathcal{T}_2(\mathcal{O})$ or $\mathcal{T}_2^*(\mathcal{H}),$
according to whether $q$ is odd or even.

Let $c$ be the tangent to $\mathcal{O}$ at $C,$ and 
$\Pi $ be an affine plane whose line at infinity is $c.$ Then there are $q$ lines of 
$\mathcal{S}$ in $\Pi ,$ they form a parallel class. Hence the corresponding set of vertices in $G$ is a deleteable set, $W.$
We can apply Lemma \ref{legalabb6} to $G$ and $W.$. The constructed graph 
$\Gamma $ is a $k$-regular graph on 
$$kq^2+(kq^2-q)=2kq^2-q$$ 
vertices.

Suppose that $\mathcal{C}$ is a $6$-cycle in $\Gamma .$ Then $\mathcal{C}$ contains two new edges, both of them join two points of $G$, because there is no
new edge among the lines. So $\mathcal{C}=(E,F,v,J,K,w),$ where the capital letters denote points of $G$, and $v$ and $w$ denote lines. Hence $F$ and $J$
are collinear points of $\mathcal{S}$. The lines $EF$ and $JK$ pass on $C,$ so the line joining the points $F$ and $J$ in $\mathrm{PG}(3,q)$ does not contain $C.$ 
This implies that $U=FJ\cap \Sigma _{\infty }$
is a point of $\mathcal{O}\setminus \{ C,F_1,F_2,\dots ,F_{q+1-k}\} .$ This is a contradiction, because 
$U$ is on the line $c=\Pi \cap  \Sigma _{\infty },$ and $c$ is the tangent to 
$\mathcal{O}$ at $C$. So the girth of $\Gamma $
is at least $7.$

Finally, by definition, the vertices corresponding to the points and lines
$(U_1,U_2,f_2,R_2,g,R_1,f_1)$ form a $7$-cycle in $\Gamma $. This completes the proof.
\end{proof}

The following theorem was first proved by Abreu et al. \cite{abr+4}. They used 
$\mathcal{Q}(4,q)$ for their construction. Now, we present a much simpler proof,
which is based on exploiting the geometric properties of ${\mathcal W}(q)$.
Note that $\mathcal{Q}(4,q)$ and ${\mathcal W}(q)$ are isomorphic if $q$ is even.
\begin{theorem}
\label{paros-jobb}
Let $q\geq 4$ be an even prime power. Then
$$n(q+1,7)\leq 2q^3+q^2+2q.$$
\end{theorem}

\begin{proof}
The classical generalized quadrangle ${\mathcal W}(q)$ of order $(q,q)$ arises from a null-polarity $\pi$ in $\mathrm{PG}(3,q)$, i.e., a polarity given by the standard form 
$X_1Y_0-X_0Y_1+X_3Y_2-X_2Y_3=0$. All points of PG$(3,q)$ are points of 
${\mathcal W}(q)$. The lines of ${\mathcal W}(q)$ are the self-conjugate lines of 
$\mathrm{PG}(3,q)$ with respect to the null-polarity $\pi$. 
The line joining the points $A=(a_0:a_1:a_2:a_3)$ and
$B=(b_0:b_1:b_2:b_3)$ is self-conjugate if and only if  
$a_1b_0-a_0b_1+a_3b_2-a_2b_3=0$. The polar plane $A^{\pi }$ of the point $A$ has equation $a_1X_0-a_0X_1+a_3X_2-a_2X_3=0$, while the pole of the plane having equation $c_0X_0+c_1X_1+c_2X_2+c_3X_3=0$ is the point $(c_1:-c_0:c_3:-c_2).$
The self-conjugate lines through the point $A$ form a pencil of lines in the plane 
$A^{\pi }.$ For $q$ even, the equations are simpler, because in this case $-1=1.$

Let $G=(V,E)$ be the Levi graph of ${\mathcal W}(q)$. Take the subset 
$W\subset V$ that consists of the following vertices: the line $X_0=X_3=0$, all of its points, and all self-conjugate lines through the $q-1$ points $\{ (0:1:t:0) \colon 0\neq t \in \mathrm{GF}(q)\} .$  By definition, $W$ is a deleteable subset.  
Let $G'$ be subgraph of $G$ induced by the set of vertices $V\setminus W.$ Then the order of $G$ is 
$$2q^3+2q^2+2q+2-\left(( 1+(q+1)+(q-1)q\right) =2q^3+q^2+2q.$$

The polar plane of the point
$E_1$ is $\Sigma_{\infty } .$ Let $\ell _{\infty ,t}$ denote the line with set of equations $\{ X_0=0, \, X_2=tX_3 \} $. Then the set of the non-deleted self-conjugate lines through $E_1$ is 
$$\mathcal{L}_1=\{ \ell _{\infty ,t}\colon t \in \mathrm{GF}(q)\} .$$
Similarly, the polar plane of the point $E_2$ is the plane $E_2^{\pi } \colon X_3=0.$ 
If $\ell _{t, \infty }$ denote the line with set of equations $\{ X_3=0, \, X_1=tX_0 \} ,$ 
then the set of the non-deleted self-conjugate lines through $E_2$ is 
$$\mathcal{L}_2=\{ \ell _{t,\infty }\colon t \in \mathrm{GF}(q)\} .$$
We define a matching of the elements of $\mathcal{L}_1$ as 
$\ell _{\infty ,t}\Longleftrightarrow \ell _{\infty ,t+1} ,$ and a matching of the elements of $\mathcal{L}_2$ as 
$\ell _{t,\infty }\Longleftrightarrow \ell _{t+1,\infty } .$

Now, we define matchings of the non-deleted affine points on the deleted affine lines.
If $t\neq 0,$ then the equation of $(0:1:t:0)^{\pi }$ is $X_0+tX_3=0,$ it intersect
$\mathrm{AG}(3,q)$ in a horizontal plane $\mathcal{H}_{1/t}$ with equation $Z=1/t.$
The self-conjugate lines in $\mathcal{H}_{1/t}$ have point at infinity 
$(0:1:t:0)^{\pi},$ so the equations of the deleted lines in $\mathcal{H}_{1/t}$
can be written as $\ell _{t,k}\colon Y=tX+k$ where $k$ runs on the elements of 
$\mathrm{GF}(q).$
For a fixed pair $(t,k)$ we define the matching of the $q$ affine points on $\ell _{t,k}$
in the following way:
$$P\colon \left( p,tp+k,\frac{1}{t} \right) \Longleftrightarrow 
P'\colon \left( p+1,tp+t+k,\frac{1}{t} \right).
$$

Apply Lemma \ref{legalabb6} to the graph $G$ and the matchings defined in the previous two paragraphs. The extended graph $\Gamma $ is $(q+1)$-regular and its
girth is at least $6.$

Suppose that $\mathcal{C}$ is a $6$-cycle in $\Gamma .$ Then, by Lemma \ref{legalabb6}, there are three possibilities:
\begin{enumerate}
\item
$\mathcal{C}=(A,b,c,D,e,f)$,
\item
$\mathcal{C}=(A,b,c,D,E,f)$,
\item
$\mathcal{C}=(a,B,C,d,E,F)$,
\end{enumerate}
where uppercase letters denote points and lowercase letters denote lines of 
$\mathcal{W}(q).$

In Case 1, the lines $b$ and $c$ belong to $\mathcal{L}_1$ or $\mathcal{L}_2.$
In the former case, $A$ and $D$ are points of $\Sigma _{\infty },$ hence 
the lines $e$ and $f$ also belong to $\mathcal{L}_1.$ In the latter case,
$A$ and $D$ are points of $E_2^{\pi },$ hence 
the lines $e$ and $f$ also belong to $\mathcal{L}_2.$ In both cases, there were 
two lines of $\mathcal{L}_i$ through the point $A,$ a contradiction.

In Case 2, the line $c$ belongs to $\mathcal{L}_1$ or $\mathcal{L}_2.$ 
Hence the point $D$ is either in $\Sigma _{\infty }$ or in $E_2^{\pi },$ and
$DE$ is a new edge through $D.$ This is a contradiction again, because there are 
no new edges through the points of the planes $\Sigma _{\infty }$ and $E_2^{\pi }.$

Finally, in Case 3, $BC$ and $EF$ are new edges, so none of these four points are in
$\Sigma _{\infty }\cup E_2^{\pi }.$ The line $d$  joins $C$ and $E,$ and $d$ is a non-deleted self-conjugate line, hence it is not horizontal. This implies that $C$ and $E$ are in
different horizontal planes, thus without loss of generality we may assume that
$C=( a,ta+k,\frac{1}{t})$ and $E=( b,sb+m,\frac{1}{s})$ where $s\neq t.$
Then $B=( a+1,ta+t+k,\frac{1}{t} )$ and 
$F=( b+1,sb+s+m,\frac{1}{s}),$ because $BC$ and $EF$ are new edges. 
The line $d$ is self-conjugate if and only if
$$a+b + \frac{ta+k}{s} + \frac{sb+m}{t}=0,$$
while the line $a$ joins $B$ and $F,$ so it is self-conjugate if and only if
$$a +1+ b + 1 + \frac{ta+t+k}{s} + \frac{sb+s+m}{t}=0.
$$
If both lines were self-conjugate, then the two equations would imply
$$\frac{t}{s}+\frac{s}{t}=0, \text{ hence } s=t,$$
a contradiction again.

Hence there is no $6$-cycle in $\Gamma .$ The order of $\Gamma $ is less than 
$M(q+1,8),$ so its girth is $7.$ The proof is complete. 
\end{proof}

As a final note, we should mention that using similar
methods, we can construct relatively small $(q,11)$-graphs and $(q+1,11)$-graphs from the Levi graph of the split Cayley generalized hexagon of order $(q,q).$
However, this will be the subject of a subsequent paper.

\noindent
{\bf Gy\"orgy Kiss:} Department of Geometry and HUN-REN-ELTE Geometric and Algebraic Combinatorics
Research Group, E\"otv\"os Lor\'and University, 1117 Budapest, P\'azm\'any
s. 1/c, Hungary; and Faculty of Mathematics, Natural Sciences and Information Technologies, University of Primorska, Glagolja\v ska 8, 6000 Koper,
Slovenia; \\
e-mail: \texttt{gyorgy.kiss@ttk.elte.hu} \\

\end{document}